\documentclass[options]{amsart}
\usepackage{latexsym,ifthen,amssymb}
\usepackage[toc,page,title,titletoc,header]{appendix}
\usepackage{enumerate}
\usepackage{graphicx}
\usepackage{bm}
\usepackage{subfigure}
\usepackage{float}
\usepackage{rotating}
\usepackage{epstopdf}
\usepackage{hyperref}

\usepackage[toc,page,title,titletoc,header]{appendix}
\usepackage{multirow}
 \usepackage{longtable}
 \usepackage{rotating}

\newtheorem{theorem}{Theorem}[section]
\newtheorem{lemma}[theorem]{Lemma}

\newtheorem{corollary}[theorem]{Corollary}
\newtheorem{proposition}[theorem]{Proposition}

\theoremstyle{definition}
\newtheorem{definition}[theorem]{Definition}

\theoremstyle{remark}
\newtheorem{remark}[theorem]{Remark}

\numberwithin{equation}{section}

\theoremstyle{noparens}
\newtheorem*{question*}{Question}
\newtheorem*{theorem*}{Theorem}

\newcommand{\uhr}{{\upharpoonright}}

\title{Upper bound on some hightness notions}

\author{Lu Liu}
\address{Department of Mathematics,
Central South University,
City Changsha, Hunan Province,
China. 410083}
\email{g.jiayi.liu@gmail.com}

\begin{document}
\def\DC{DC}
\def\wt{\text{wt}}
\def\h{\hat}
\def\dominate{dominate}
\def\t{\tilde}
\def\mcal{\mathcal}
\def\m{m}
\def\msf{\mathsf}
\def\S{S}
\def\majorize{majorize}
\def\geqs{\geq_s}
\def\v{\vec}

\maketitle

\begin{abstract}
We give upper bound for several
highness properties in computability randomness
theory.
First, we prove that
discrete covering property does not imply
the ability to compute a 1-random real,
answering a question of Greenberg, Miller and Nies.
This also implies that an infinite  set of incompressible strings
 does not necessarily extract a  1-random real.
Second, we
prove that given a homogeneous binary tree that does not admit
an infinite computable path, a sequence of bounded martingale
whose initial capital tends to zero,
 there exists
a martingale $S$ majorizing infinitely many of
them such that $S$ does not compute an infinite path
of the tree. This implies that 1)
 High(CR,MLR) does not imply PA-completeness,
 answering a question of Miller;
2) $\leq_{\msf{CR}}$ does not imply $\leq_T$,
answering a question of Nies.
 The proof of the second
 result suggests that the coding power
 of the universal c.e. martingale
 lies in its infinite variance.

\end{abstract}

\section{Introduction}

The relation between different complexity notion,
characterize the complexity of a given object
are central topic of mathematical logic.
It has been established that many lowness
(for random) notion 
coincide.
Therefore we naturally wonder:
does a similar picture shows up for highness notions?
Our goal is to  compare the following two pairs of highness notions:
1) PA degree (the Turing degree that compute
a complete Peano Arithematic)
versus High(CR,MLR), the oracle  relative
to which every computably random real is
1-random; 
2) the Turing degree that compute a 1-random real
versus
 the discrete covering property.
 In each case, we separate the former notion 
 from the later one.

In computability theory,
there are several ways to say
an object is of low complexity.
Let $\mcal{C}\subseteq 2^\omega$
be a relativizable complexity notion,
if $\mcal{C}^A=\mcal{C}$, then this is
seen as an evidence that $A$ is of low complexity.
For example, let $\mcal{C}^A$  be
1-random relative to $A$, then
$\mcal{C}^A=\mcal{C}$ means every 1-random real
is still 1-random relative to $A$ (low for 1-random).
The second way is by Kolmogorov complexity.
$A$ is $K$-trivial iff
$\forall n K(A\uhr_n)\leq K(n)+c$ for some constant $c$.
The third way is to say that some robust
object compute $A$. For example,
 $A$ is a base for ML-randomness if
 $A$ is Turing reducible to some 1-random
 relative to $A$ real.
Nies\cite{nies2005lowness}
and Hirschfeldt, Nies and Frank
\cite{hirschfeldt2007using} showed that these complexity notions coincide 
(see \cite{Nies2009Computability}\cite{Downey2010Algorithmic} for
background knowledge of computability randomness theory).

Highness notion
can be described in a similar fashion as lowness notion
but in an opposite  direction.
For example, in contrast with low for 1-random
is the ability to compute some 1-random.
A weaker notion in this respect is that $A$
has more efficient way to describe non-1-random real
than the computable degree.
In this aspect, Franklin, Stephan, and Yu
\cite{franklin2011relativizations} studied the following highness
notion.
An oracle $D$ is \emph{high for computable
randomness vs ML-randomness}
if every real that is computably random relative
to $D$ is $1$-random.
We denote this property as High(CR,MLR).
It it known that  High(CR,MLR) implies
the ability to compute a 1-random real \cite{franklin2011relativizations}.
While conversely, computing a 1-random real is far from
being High(CR,MLR) (see \cite{greenberg2019highness} figure 1).
Meanwhile, \cite{franklin2011relativizations} observed
that every PA degree is
 High(CR,MLR). Most of highness properties
 of this type, except High(CR,MLR), has been well understood.
A fundamental question about High(CR,MLR) by Miller,
is that whether there is
an actual gap between PA completeness and High(CR,MLR)
(see \cite{Nies2014logic}).
Given a set $D\subseteq 2^{<\omega}$,
in this paper,
a \emph{martingale}
(\emph{supermartingale} resp)  on $D$ is a function
$\S:D\rightarrow \mathbb{R}^{\geq 0}$
such that for every $\rho\in D$ with $\rho0,\rho1\in D$,
$2\S(\rho) =
(\geq \text{resp}) \S(\rho 0)+\S(\rho 1)$.
A nice characterization of High(CR,MLR)
given by Kastermans, Lempp, and Miller
(or Bienvenu and Miller \cite{bienvenu2012randomness}) is that
$D\in $High(CR,MLR) if and only if
$D$ compute a martingale $S$ succeeding on every non-1-random
 where
$S$ succeed on $X\in 2^\omega$ if
$\limsup\limits_{n\rightarrow\infty}S(X\uhr_n)=\infty$.
In Theorem \ref{dcth0} we construct such a martingale
that does not compute any PA degree, thus answer the question
positively.
On the other hand, \cite{greenberg2019highness} proved
that every $D\in $High(CR,MLR) compute
a DNR$_h$ where $h$ is some slow growing
computable increasing function. Thus the current upper bound and
lower bound of High(CR,MLR) are pretty close.
We say $C$ is Low for computable random
iff  every computably random real is computably random relative to $C$.
 Relativizing a lowness notion give rise to a 
 reducibility notion.
 We write $C\leq_{\msf{CR}} D$
if $\msf{CR}^D\subseteq \msf{CR}^C$ where
$\msf{CR}^A$ is the set of
computably random reals relative to $A$.
Nies \cite{nies2005lowness} proved that
 $C$ is  Low for computable random if and only if
 $C$ is computable. 
 Therefore,
Nies \cite{Nies2014logic} naturally asked whether $\leq_{\msf{CR}}$ implies
$\leq_T$. Miyabe confirms  the analog
 question for Schnorr randomness.
 We show that this is not the case
 for computable randomness (Corollary \ref{dccoro1}),
 answering the question negatively.

Another type of highness notion is the ability to compute
the solution of certain combinatorial object. PA completeness
falls in this category since $A$ is of PA degree if and only if
for every computable infinite binary tree
$T$, $A$ compute an infinite path through $T$. From the view of
reverse math,
the tree $T$ is an instance, an infinite path through $T$ is
called its solution.
 Greenberg,Miller and Nies \cite{greenberg2019highness} defined a highness notion
in a similar fashion.
Given a set $X\subseteq \omega\times\omega$,
the \emph{weight} of $X$, written as $\wt(X)$,
is $\sum\limits_{(n,m)\in X}2^{-n}$.
The discrete covering property is defined as following.
\begin{definition}[\cite{greenberg2019highness}]
An instance of \emph{Discrete Covering}
(henceforth $\msf{\DC}$) is a
set $X\subseteq \omega\times\omega$
such that $\wt(X)<\infty$.
A $\msf{\DC}$ solution to $X$ is
a set $Y\supseteq X$ such that $\wt(Y)<\infty$.
An oracle $D$ has the \emph{discrete covering property}
if $D$ compute a solution to every
c.e. $\msf{DC}$ instance.

\end{definition}
They
asked if an oracle $D$ compute $\msf{\DC}$
solution to all c.e. $\msf{\DC}$ instance,
then does $D$ necessarily compute a  $1$-random
real?
According to the experience in reverse math,
two naturally arised combinatorial notion
could hardly coincide with each other.
It is even harder to approach $\msf{WWKL}$
(the ability to compute 1-random real)
with combinatorial notion
involving subset.
Where the $\msf{DC}$ problem involves subset notion
since $X$ is a solution and $Y\supseteq X\wedge \wt(Y)<\infty$
implies $Y$ is also a solution.
 Chong,Li,Liu,Yang\cite{chong2019strength} shows that even the tree theorem for pairs
 (which incorporate tree combinatorics in itself) does not imply
 the Weak Weak K\"{o}nig's Lemma.
Therefore it is foreseeable that discrete covering
does not imply the ability to compute 1-random real
which is confirmed in Theorem \ref{dcth1}.
An interesting note is that a
discrete covering oracle compute
an infinite set of incompressible strings
(see Proposition \ref{dcprop0}).
Thus together with Theorem \ref{dcth1} it means
that an infinite set of incompressible string
does not necessarily extract a 1-random real.
Another consequence is that discrete covering property
does not imply (strong) continuous
covering property since the latter implies
computing a 1-random real (see Corollary \ref{dccoro0}).

It is interesting to wonder if we can view
High(CR,MLR) as a reverse math problem.
To say that, it usually means 1) The set of instance
is a closed set (in most cases an effectively closed set);
2) Given an instance $X$, the solution to $X$
is "almost" a $\Pi_1^{0,X}$ class.
For example, given an infinite tree $T$,
the set of infinite path through $T$
is a $\Pi_1^{0,T}$ class uniformly in $T$.
But given a coloring of pairs $C:[\mathbb{N}]^2\rightarrow k$,
the collection of infinite homogeneous set of $C$
is not $\Pi_1^{0,C}$, but a $\Pi_1^{0,C}$ class
intersects with countably many densely open set.
Note that the set of martingale $S$ with  $S(\bot)\leq c$ is a
$\Pi_1^{0,c}$ class uniformly in $c$. Therefore,
can we define the majorize martingale problem
(henceforth $\msf{MM}$ in this paper) as following?
The instance of $\msf{MM}$ is a martingale $S$,
the solution to $S$ is a martingale $S^*$ such that
$S^*\geq S$. However, this restriction lose the essence
of High(CR,MLR). \cite{greenberg2019highness} shows that every
martingale majorizing the optimal c.e. submartingale
compute a PA degree.
On the other hand, we guess that this coding power of
the optimal c.e. submartingale lies in its infinite
variance. i.e., We guess that for every
martingale
(not necessarily c.e.) with bounded variance (but not necessarily bounded)
admit a majorization that does not compute any PA degree
(see section \ref{dcth0}).
We deal with discrete covering property
in section \ref{dcsec0}
and High(CR,MLR) in section \ref{dcsec1}.

\textbf{Notations.}
We write $(\Psi^\tau\uhr_N)\downarrow$ if
$\Psi^\tau(n)\downarrow$ for all $n\leq N$;
and $(\Psi^\tau\uhr_N)\uparrow$ if $\Psi^\tau(n)\uparrow$
for some $n\leq N$.
For a string $\rho \in 2^{<\omega}$, we let $[\rho]^{\preceq}=\{\sigma\in 2^{<\omega}: \sigma\succeq \rho\}$,
$[\rho] = \{X\in 2^\omega:X\succ \rho\}$;
similarly, for $V\subseteq 2^{<\omega}$, let $[V]^{\preceq} = \{\sigma\in 2^{<\omega}:
 \sigma\succeq \rho\text{ for some }\rho\in V\}$,
 $[V] = \{X\in 2^\omega:X\succ\rho\text{ for some }\rho\in V\}$.
For $X\in 2^{<\omega}\cup 2^\omega$, let
$X\uhr_n = X(0)\cdots X(n-1)$, $X\uhr_n^m = X(n)X(n+1)\cdots X(m)$;
for $Q\subseteq 2^\omega\cup 2^{<\omega}$, $Q\uhr_n = \{X\uhr_n:X\in Q\}$.
We use $\bot$ to denote the empty string.
For an effective object such as a c.e. set $D$, a $\Pi_1^0$ class
$Q$, we write $D[t],Q[t]$ to denote the object computed by time $t$.
For a set $V\subseteq 2^{<\omega}$,
we write $\m(V)$ to denote the Lebesgue measure of $[V]$
in $2^\omega$.

\section{Discrete covering vs 1-randomness}
\label{dcsec0}

Generally speaking, a problem $\msf{P}$ admit
$\Pi_1^0$ class avoidance of a property, if every
non empty $\Pi_1^0$ class of $\msf{P}$-instance
admit a member and a solution $G$ of that member
such that $G$ avoid that property. Note that
a c.e. $\msf{DC}$ instance $D$ with, say
$\wt(D)<1$, can be approximated
by a non empty $\Pi_1^0$ class in the following way:
$Q = \{X:\text{ For every }t, X\supseteq D[t]
\text{ and }\wt(X)\leq 1\}$. Clearly $D$ itself
is a member of $Q$ and for every
$X\in Q$, every $\msf{DC}$ solution $Y$ of $X$,
$Y$ is a solution to $D$.
Our Lemma \ref{dclem2}
actually shows  that $\msf{DC}$ problem admit $\Pi_1^0$
class avoidance of $1$-randomness.
We write $X\subseteq^* Y$
if $X\setminus Y$ is finite.
\begin{theorem}\label{dcth1}
There exists a $G\subseteq \omega\times\omega$
with $\wt(G)<\infty$
such that for every c.e. $\msf{\DC}$ instance
$D$, $D\subseteq^* G$ and $G$
does not compute any $1$-random real.
\end{theorem}

An oracle $D$ has
\emph{(strong) continuous
covering property} if for every
computable tree $T$ with
$\m([T])>0$ there exists a $D$-computable
subtree $T'$ of $T$ with no leaf such that
$\m([T'])>0$.
\cite{greenberg2019highness}
 proved that
(strong) continuous
covering property implies computing a 1-random
real. Thus we have:

\begin{corollary}\label{dccoro0}
Discrete covering property does not imply
(strong) continuous
covering property.
\end{corollary}

An interesting observation is that:
\begin{proposition}\label{dcprop0}
There is a c.e. $\msf{DC}$ instance
$D$ such that every $\msf{DC}$ solution
$ X$ of $D$ compute an infinite
set $W$ of incompressible strings.
i.e., $W\subseteq \{\rho:[\rho]\nsubseteq [U_n]\}$
for some $n$
where $(U_n:n\in\omega)$ is the universal ML-test.
\end{proposition}
\begin{proof}
For $e\in\omega$, let $D_e$ denote the finite subset
of $\omega\times\omega$ whose canonical index is $e$.
Let $h:\omega\times 2^{<\omega}\rightarrow
\omega$ be such a computable function that:
\begin{itemize}
\item For every $(n,\rho)\in \omega\times 2^{<\omega}$,
$\wt(D_{h(n,\rho)}) = 2^{-|\rho|}$
\item $D_{h(n,\rho)}$ are mutually disjoint.
\end{itemize}
Let $(U_n:n\in\omega)$ be the universal ML-test.
We view each $U_n$ as a prefix free c.e. subset of $2^{<\omega}$.
Let $$
D=\bigcup\limits_{n\in\omega,\rho\in U_n}D_{h(n,\rho)}.
$$
Note that for every $n$,
$\wt(\bigcup\limits_{\rho\in U_n}D_{h(n,\rho)})
 = \m(U_n) \leq  2^{-n}$.
 Thus $\wt(D)<\infty$. It is clear that $D$ is c.e.
 Suppose $X\supseteq D$ with $\wt(X)<\infty$.
 Note that this means that there exists an $m^*$
 such that for every $n$, there exists a $\rho\in 2^n$
 such that $D_{h(m^*,\sigma)}\nsubseteq X$
 for all $\sigma\preceq\rho$.
 Thus let
 $$
 W = \big\{
 \rho: \forall \sigma\preceq\rho\  D_{h(m^*,\sigma)}\nsubseteq X
 \big\}.
 $$
Clearly $W$ is $X$-computable and infinite.
Moreover, if $\exists \sigma\preceq\rho\ \sigma\in U_{m^*}$,
then $D_{h(m^*,\sigma)}\subseteq X$, which
means $\rho\notin W$. Thus
$W\subseteq\{\rho:
[\rho]\nsubseteq [U_{m^*}]\}$ and we are done.

\end{proof}
An infinite set of incompressible$_K$ string
is an infinite set $W\subseteq 2^{<\omega}$
such that for some $c$, $K(\sigma)>|\sigma|-c$
for all $\sigma\in W$. It is well known
that $W$ is incompressible$_K$ if and only if
$W\subseteq
\{\rho:[\rho]\nsubseteq U_n\}$ for some $n$.
Combine Theorem \ref{dcth1} and Proposition \ref{dcprop0}
we have:
\begin{corollary}\label{dccoro3}
There exists an infinite set $W$ of incompressible$_K$
strings such that $W$ does not compute any 1-random real.
\end{corollary}
Meanwhile,
Corollary \ref{dccoro3} can also be interpreted as
 a $1$-enumeration
 (which is a function $g:\omega\ni n\mapsto
 g(n)\in 2^n\cap T_n$)
 of the tree $T_n = \{\rho:[\rho]\nsubseteq [U_n]\}$
 does not necessarily compute an infinite path through $T_n$.
The remaining of this section will prove Theorem \ref{dcth1}.
\subsection{Proof of Theorem \ref{dcth1}}
The frame work of the proof follows the usual
line in computability theory. We build a sequence of 
conditions each forces some requirement and 
thus the oracle $G$ will satisfy all needed requirements.
We firstly define the condition we use then
we focus on how to force a given requirement
(Lemma \ref{dclemmain}).

Fix a computable order on $\omega\times\omega$.
Then we can think of a set $X\subseteq \omega\times\omega$
as a binary sequence $X\in 2^\omega$ and
Conversely, for every $\rho\in 2^{<\omega}$
or $X\in 2^\omega$, we can think of
$\rho,X$ as subset of $\omega\times\omega$.
Therefore it makes sense to write $\wt(\rho),
\rho\cap X,\rho\supseteq \sigma,\rho\setminus X$.
We say $\h{Q}$ \emph{\dominate} $\t{Q}$
if for every $Y\in \h{Q}$
there exists an $X\subseteq Y$ such that $X\in \t{Q}$.
For a class $\h{Q}\subseteq 2^\omega$,
 let
$$\wt(\h{Q}) = \inf\limits_{X\in\h{Q}}\wt(X)
\text{ and }\wt(\h{Q}) = \infty \text{ if }\h{Q}=\emptyset.$$
In this section except section \ref{dcseccomb},
all reals appearing as parameters are
computable.

In this section, a
 \emph{condition}
is a tuple $(\tau,Q,a)$ such that
\begin{itemize}
\item $\tau\in 2^{<\omega}$;
\item $Q\subseteq 2^\omega$ is a   $\Pi^0_1$
class  such that $\wt(Q)<\infty$;

\item For every $X\in Q$, $\tau\supseteq X\uhr_{|\tau|}$
and $\wt(\tau\setminus X)<a$.

\end{itemize}
Intuitively, a condition is seen as
a  $\Pi^0_1$ class of candidates of the 
oracle $G$ we construct, namely
$$
\big\{
Y\succ\tau:
\text{There exists an }X\in Q,
Y\supseteq X,
\wt(Y\setminus X)\leq a
\big\}.
$$
We use $(\tau,Q,a)$ to denote this set of candidates.
It is not necessary that for every $X\in Q$,
$\wt(X)<\infty$.
A condition $(\tau',Q',a')$
\emph{extends} a condition $(\tau,Q,a)$
(written as $(\tau,Q,a)\leq (\tau',Q',a')$) if
$(\tau',Q',a')\subseteq (\tau,Q,a)$.
The requirement is simply
$$\mcal{R}_{\Psi,m}:
\Psi^G\in [U_m]\text{ or }\Psi^G \text{ is not total.}$$
A condition $(\tau,Q,a)$ forces $\mcal{R}_{\Psi,m}$
if every $G\in (\tau,Q,a)$ satisfies
$\mcal{R}_{\Psi,m}$.

It's easy to see that:
\begin{lemma}
For every condition $(\tau,Q,a)$,
every c.e. $\msf{DC}$ instance $D$,
there exists an extension $(\tau,Q',a')$
of $(\tau,Q,a)$ such that
for every $Y\in (\tau,Q',a')$,
$D\subseteq^*Y$.
\end{lemma}
\begin{proof}
Simply let $n>|\tau|$ be sufficiently large
so that $\wt(D\setminus (\omega\times\omega\uhr_n))<\delta$
\footnote{Recall that we think of $\omega\times\omega$
as an element in $2^\omega$ so $\omega\times\omega\uhr_n$
is its initial segment of length $n$.}
with $\delta$ being sufficiently small
such that for some $X^*\in Q$ with  $\wt(X^*)<\infty$,
we have $\wt(\tau\setminus X^*)< a'<a-\delta$.
Let
$$Q' = \big\{X'\succ X^*\uhr_{|\tau|}:\text{There exists an }X\in Q
\text{ such that }\wt(X'\setminus X)\leq \delta,
X'\supseteq X\cup (D\setminus (\omega\times\omega\uhr_n))\big\}.$$

To see $(\tau,Q',a')\subseteq (\tau,Q,a)$,
fix a $Y\in (\tau,Q',a')$ with $X'$ being a witness.
i.e., $Y\supseteq X'\wedge\wt(Y\setminus X')\leq a'$.
By definition of $Q'$,
there exists an $X\in Q$ such that $X'\supseteq X$
and $\wt(X'\setminus X)\leq \delta$.
Therefore $$\wt(Y\setminus X)=\wt(Y\setminus X')+\wt(X'\setminus X)
\leq a'+\delta<a.$$
Meanwhile $X^*\cup (D\setminus (\omega\times\omega\uhr_n))\in Q'$,
thus $\wt(Q')<\infty$. It's trivial to verify that
for every $Y\in (\tau,Q',a')$, $D\subseteq^*Y$.
\end{proof}
Now it remains to prove the following:
\begin{lemma}\label{dclem2}
Every condition admit an extension forcing
a given requirement $\mcal{R}_{\Psi,m^*}$.
\end{lemma}

Using Lemma \ref{dclem2}, we can build 
a sequence of  conditions $d_0\geq d_1\geq \cdots$
where $d_s=(\tau_s,Q_s,a_s)$ such that every 
requirement is forced by one of these conditions
and $d_0$ is weight-bounded in the sense that 
for some $b<\infty$, 
$\wt(X)<b$ for all $X\in Q_0$.
By compactness, $G=\cup_s \tau_s$ is contained in all conditions and
 thus satisfy all requirements. We also have $\wt(G)<\infty$
 since $G\in d_0$ and $d_0$ is weight-bounded.
Lemma \ref{dclem2} will follows from the following:
\begin{lemma}\label{dclemmain}
For a condition $(\tau,Q,a)$
and a $\varepsilon>0$, there exists
a tuple
 $(\tau^*,Q^*,a^*,V^*)$ such that
\begin{itemize}

\item $\tau^*\succeq\tau$ and $Q^*$ \dominate\ $Q$;
\item $\wt(Q^*)+a^*<a$;
\item For every $X\in Q^*$,
$\tau^*\supseteq X\uhr_{|\tau^*|}$ and $\wt(\tau^*\setminus X)<a^*$;

\item $V^*\subseteq 2^{<\omega}$
is c.e. with $\m(V^*)<\varepsilon$. Moreover, its index can be computed from
the index of $Q$, an upper bound of $\wt(Q)$,
and $\tau,a,\varepsilon$.

\item Either for some $n$, $\Psi^{\tau^*}\uhr_n\in V^*$
(in which case we say $(\tau^*,Q^*,a^*,V^*)$
forces $\mcal{R}_{\Psi,m^*}$ in a $\Sigma_1^0$ way);
 Or
for every $X\in Q^*$,
every $Y\supseteq X$ with $\wt(Y\setminus X)\leq a^*$,
$\Psi^Y$ is non total
(in which case we say $(\tau^*,Q^*,a^*)$
forces $\mcal{R}_{\Psi,m^*}$ in a $\Pi_1^0$ way).
\end{itemize}
\end{lemma}
In this case,
$(\tau^*,Q^*\cap \{X:\wt(X)\leq \wt(Q^*)+\delta\},a^*)$
is a condition extending $(\tau,Q,a)$
and forces $\mcal{R}_{\Psi,m^*}$
if $[V^*]\subseteq [U_m]$
(where $\delta$ is sufficiently small so that $\wt(Q^*)+a^*+\delta<a$).
To see this, let $Y\in (\tau^*,Q^*\cap \{X:\wt(X)\leq \wt(Q^*)+\delta\},a^*)$
with $X^*\in Q^*\cap \{X:\wt(X)\leq \wt(Q^*)+\delta\}$ being a witness. That is:
$$Y\supseteq X^*, Y\succeq\tau^*, \wt(Y\setminus X^*)\leq a^*.$$
 Since $Q^* $ dominate $Q$,
there exists $X\subseteq X^*$,  such that $X\in Q$.
Then
we have: $$\wt(Y\setminus X)
=\wt(Y\setminus X^*)+\wt(X^*\setminus X)
\leq a^*+\wt(X^*)\leq a^*+\wt(Q^*)+\delta<a.$$
i.e., $Y\in (\tau,Q,a)$. Therefore $(\tau^*,Q^*\cap \{X:\wt(X)\leq \wt(Q^*)+\delta\},a^*)$
extends $(\tau,Q,a)$.

Now it remains to prove Lemma \ref{dclemmain}.
For a condition $(\tau,Q,a)$,
note that we can select
a $X\in Q$ with $\wt(X)<\infty$ and
extends $\tau$ to $\tau'\supseteq X\uhr_{|\tau'|}$ so that
$\wt(X\setminus (\omega\times\omega\uhr_{|\tau'|}))$
is very small compared to $a-\wt(\tau'\setminus X)$
(simply let $\tau'\uhr_{|\tau|}^{|\tau'|-1} =
X\uhr_{|\tau|}^{|\tau'|-1}$).
Therefore, for convenience, we
prove Lemma \ref{dclemmain} for $(\bot,Q,3)$
and assume that $\wt(Q)<1$ (but make no  assumption on $\varepsilon$).
Before that, we firstly
establish some combinatorics.
\subsection{Small set in partition-large collection}
\label{dcseccomb}

It is corollary \ref{dclemcomb2} that we are going to use
(actually we will only use it for $k=2$).
The \emph{Hamming distance} between
$\rho,\sigma\in 2^n$, denoted as
$d_H(\rho,\sigma)$, is
$|\{m< n: \rho(m)\ne\sigma(m)\}|$;
for a set $\mcal{A}\subseteq 2^n$,
$d_H(\rho,\mcal{A}) = \min\limits_{\sigma\in\mcal{A}}d_H(\rho,\sigma)$;
for a $Z\subseteq N$,
 $d_H(\rho,Z)=|Z\Delta \{n<N:\rho(n)=1\}|$.
Recall that   a sequence of random variable
$x_0,x_1,\cdots$ is $o_p(n)$ if for every $c>0$,
$\lim\limits_{N\rightarrow\infty} \mathbb{P}(|x_N|<cN)=1$.
Applying Theorem 3 of \cite{katona1975hamming}, we have the
following.
\begin{theorem}\label{dclemcomb0}

Let $c>0$ and  $N$ be very large;
let $\mcal{A}\subseteq 2^N$ satisfy $|\mcal{A}|/2^N\geq
c$; let $\rho$ be uniformly random in $2^N$.
Then,
 $d_H(\rho,\mcal{A}) = o_p(N)$.

\end{theorem}

A collection $\mcal{A}\subseteq \mcal{P}(N)$
is $k$-\emph{large} in $N$
if for every $k$-partition $N_0,\cdots,N_{k-1}$
of $N$, there exists an $i\in k$ such that $N_i\in \mcal{A}$.
$\mcal{A}$ is \emph{upward closed }  if for
every $A\supseteq B$, $B\in \mcal{A}$
implies $A\in \mcal{A}$.
Let $N$ be very large
and  $\delta>0$ be a constant.
\begin{theorem}\label{dclemcomb}
Suppose $\mcal{A}\subseteq\mcal{P}(N)$ is $k$-large
in $N$ and upward closed;
let $Z$ be a uniformly random
$N/k$-element subset of $N$,
we have that:
with high probability (as $N
\rightarrow\infty$),
there exists a $A\subseteq N$ with $|A|<\delta N$
such that
$A\cup Z\in \mcal{A}$.
\end{theorem}
\begin{proof}
We only prove for $k=2$. The general case
follows by coding $k^N$ into $2^{N'}$ in a bit wise
fashion.
Firstly,
let $$
\mcal{A}' = 
\big\{\rho\in 2^N:\{n:\rho(n)=1\}\in \mcal{A}
\big\}.
$$

Since $\mcal{A}$ is $2$-large,
it's easy to see that
$\lim\limits_{N\rightarrow\infty}
|\mcal{A}'|/2^N\geq 1/2$.
Let $\rho$ be uniformly random in
$2^N$ and let 
$x = \min\{d_H(Z,\rho):Z\subseteq N,|Z|=N/2 \}$.
Let $Z'$ be selected uniformly
randomly among
$\{Z\subseteq N:|Z|=N/2\wedge d_H(Z,\rho)=x \}$.
Clearly $Z'$ is a uniformly random subset of $N$
and
 $d_H(Z',\rho)=x=o_p(N)$.
 Therefore, by Theorem \ref{dclemcomb0}, $d_H(Z',\mcal{A}')=o_p(N)$.
 This means that with high probability,
 there exists a $\sigma\in \mcal{A}'$
 such that $|Z'\Delta\{n<N:\sigma(n)=1\}|<\delta N$.
 But $\mcal{A} $ is upward closed, so
 $Z'\cup A\in\mcal{A}$ where
 $A=\{n<N:\sigma(n)=1\}\setminus Z'$.
 Thus we are done.

\end{proof}

It follows directly from \ref{dclemcomb} that:
\begin{corollary}\label{dclemcomb2}
For every $\delta>0$,  every $j\in\omega$,
every $j$ many collections $\mcal{A}_0,\cdots,
\mcal{A}_{j-1}$ with each being $k$-large in
$N$ and upward closed (where $N$ is sufficiently large),
there exists a subset $Z$ of $N$ with $|Z|<(1/k+\delta)N$
such that $Z\in \bigcap\limits_{j'<j}\mcal{A}_{j'}$.
\end{corollary}

\subsection{Proof of Lemma \ref{dclemmain}}

For $a>0, n\in\omega, V\subseteq 2^n$,
let
\begin{align}\label{dceq6}
Q(V,a) = \big\{
Y\in 2^\omega:
&\text{There exists an }X\in Q\text{ with } X\subseteq Y;\\ \nonumber
&\text{ For every }Z\supseteq Y\text{ with }
\wt(Z\setminus Y)\leq a
\wedge (\Psi^Z\uhr_n)\downarrow
\text{ we have that }\Psi^Z\uhr_n\notin V
\big\}.
\end{align}

Clearly $Q(V,a)$ is a $\Pi_1^0$ class uniformly in
$V$ and the index of the computable representation of $a$.
Note that if there is no desired tuple to force
$\mcal{R}_{\Psi,m^*}$ in a $\Sigma_1^0$ way,
then  $\wt(Q)\geq \wt(Q(V^*,a))$
for all $a<  3/2$. To see this,
suppose $\wt(Q)<\wt(Q(V^*,a))$
for some $a< 3/2$, then there exists
an $X\in Q$ with $\wt(X)<\wt(Q(V^*,a))\wedge \wt(X)<1$,
which means $X\notin Q(V^*,a)$.
By definition of $Q(V^*,a)$,
there exists a $\tau\supseteq X\uhr_{|\tau|}$
with $\wt(\tau\setminus X)\leq a$
such that $(\Psi^\tau\uhr_n)\downarrow\in V^*$ for some $n$.
Therefore, the tuple
 $(\tau,[X\uhr_{|\tau|}]\cap Q,3/2,V^*)$ is the desired
 extension of $(\bot,Q,3)$
 forcing $\mcal{R}_{\Psi,m^*}$ in a $\Sigma^0_1$ way.
\begin{remark}
Intuitively, $\wt(Q(V,a))$ can be seen
as a measure of how much effort $Q$ has to make in
order to ensure that $\Psi^G\notin [V]$,
the larger $\wt(Q(V,a))$ is the harder it is for
$Q$ to  ensure $\Psi^G\notin [V]$.
The c.e. set $V^*$  will try to enumerate
 $V'$ so that at some time $t$,
 $\wt(Q(V',a)[t])-\wt(Q[t])$ is considerably larger than
 $\m(V')$.
 By the above observation, this will force $\wt(Q)$ to increase
 (compared to $\wt(Q[t])$)
 considerably large.
The key Lemma \ref{dclem3} shows that for every sufficiently
 small $\Delta\lambda>0$,
 we can force $\wt(Q)$ to increase by
 at least $\Delta\lambda/2$ through enumerating
 a "small" set $V'$ into $V^*$ where
 $\m(V')<(\frac{2}{3})^{\frac{1}{\Delta\lambda}}$.
 \end{remark}
 We firstly make some observations on
 $Q(V,a)$.
For convenience, whenever we write $Q(V,a)$,
 it automatically implies that $a$ is irrational
 (see footnote below for the reason).

\begin{lemma}\label{dclem0}
Let $\rho\in 2^{<\omega}$,
$V_0,V_1\subseteq 2^n, V\subseteq 2^m$;
let $a,\lambda,\lambda'>\delta>0$.
\begin{enumerate}
\item If $a'\geq a$, then $Q(V,a')\subseteq Q(V,a)$;
If $V'\supseteq V$, then $Q(V',a)\subseteq Q(V,a)$.
\item Suppose $\wt([\rho]\cap Q(V_i,a))<\lambda+\lambda'$
for $i\in 2$ and $\wt(\rho)> \lambda-\delta$,
then we have:
$\wt([\rho]\cap Q(V_0\cup V_1,a-2\lambda'))<\lambda+3\lambda'$.

\item Suppose $a\leq 3/2$, $V'=[V]^\preceq\cap 2^n$ and

$\wt([\rho]\cap Q(V',a))<3/2<\wt([\rho]\cap Q(V,a))$.
Then  there    exists a desired tuple
 forcing $\mcal{R}_{\Psi,m^*}$ in a $\Pi_1^0$ way.

\item Suppose $a\leq 3/2$, $\wt( Q(2^n,a))<3/2$,
then  there   exists a desired tuple
 forcing $\mcal{R}_{\Psi,m^*}$ in a $\Pi_1^0$ way.
\end{enumerate}

\end{lemma}
\begin{remark}
Intuitively, item (2) says that if we 
can avoid each $V_0,V_1$ with a small price,
namely $\lambda'$, then 
we can avoid them together with a small price,
namely $3\lambda'$.
The intuition for item (3) is that 
if we cannot avoid $V$ with a certain price, then 
we can not  avoid $ [V]^\preceq\cap 2^n$
since if $\Psi^G$ is total, then 
$\Psi^G\in [V]\Leftrightarrow \Psi^G\in [[V]^\preceq\cap 2^n]$.

\end{remark}
\begin{proof}
Item (1) is trivial and its proof is skiped.

Proof of item (2):
Let $X_i\in [\rho]\cap Q(V_i,a)$ with
$\wt(X_i)<\lambda+\lambda'$
and let $X = X_0\cup X_1$.
We show that $X$ is a $\msf{DC}$ instance
avoiding $V_0\cup V_1$
(in the sense of (\ref{dceq6})).
Clearly $X\in [\rho]$.
Meanwhile, $$\wt(X)=\wt(X_0)+\wt(X_1)-\wt(X_0\cap X_1)
<2\lambda+2\lambda'-\lambda+\delta
<\lambda+3\lambda'.$$
Note that for every $Y\supseteq X$ with
$\wt(Y\setminus X)\leq a-2\lambda'$,
it holds that
\begin{align}
\wt(Y\setminus X_i)
&=\wt(Y\setminus X)+\wt(X\setminus X_i)
\\ \nonumber
&\leq a-2\lambda'+\wt(X_{1-i}\setminus X_i)\\ \nonumber
&=a-2\lambda'+\wt(X_{1-i}\setminus (X_0\cap X_1))
\\ \nonumber
&\leq a-2\lambda'+\wt(X_{1-i}\setminus \rho)
\\ \nonumber
&< a-2\lambda'+\lambda+\lambda'-\lambda+\delta
\\ \nonumber
&<a
\end{align}

This means that for every $Y\supseteq X$ with
$\wt(Y\setminus X)\leq a-2\lambda'$,
it holds that $(\Psi^Y\uhr_n)\downarrow
\rightarrow\Psi^Y\uhr_n\notin V_0\cup V_1$.
i.e., $X\in Q(V_0\cup V_1,a-2\lambda')$.
Thus 

$\wt([\rho]\cap Q(V_0\cup V_1,a-2\lambda'))<\lambda+3\lambda'$.

\ \\

Proof of item (3):
Let $X^*\in [\rho]\cap Q(V',a)$
with $\wt(X^*)<3/2$.
Clearly $X^*\notin Q(V,a)$.
Therefore there exists a
$\tau\supseteq X^*\uhr_{|\tau|}$ with $\wt(\tau\setminus X^*)<a$
\footnote{Since $a$ is irrational and $\wt(\tau\setminus X^*)\leq a$
is  rational, therefore $\wt(\tau\setminus X^*)<a$.}
such that $(\Psi^\tau\uhr_m)\downarrow\in V$. We show that $(\tau,
[X^*\uhr_{|\tau|}]\cap Q(V',a),a)$ is a desired tuple
 forcing $\mcal{R}_{\Psi,m^*}$ in a $\Pi_1^0$ way.
Clearly $$\wt([X^*\uhr_{|\tau|}]\cap Q(V',a))+a\leq \wt(X^*)+3/2<3,$$
 and $[X^*\uhr_{|\tau|}]\cap Q(V',a)$
\dominate\ $Q$. Fix a $X\in [X^*\uhr_{|\tau|}]\cap Q(V',a)$,
a $Y\supseteq X$ with $Y\succ\tau\wedge \wt(Y\setminus X)\leq a$.
Note that $(\Psi^Y\uhr_n)\downarrow\rightarrow\Psi^Y\uhr_n\notin V'$.
But $\Psi^Y\uhr_m\in V$ and $V'=[V]^\preceq\cap 2^n$.
This means $(\Psi^Y\uhr_n)\uparrow$.

Proof of item (4):
We show that $(\bot, Q(2^n,a),a)$
is the desired tuple
 forcing $\mcal{R}_{\Psi,m^*}$ in a $\Pi_1^0$ way.
Fix a $X\in  Q(2^n,a)$.
Clearly for every $Y\supseteq X$
with $\wt(Y\setminus X)\leq a$,
$(\Psi^Y\uhr_n)\uparrow$.
Meanwhile $\wt( Q(2^n,a))+a<3$.

\end{proof}
The following is the key Lemma
and is the only part of the proof concerning
the combinatorics in section \ref{dcseccomb}.
\begin{lemma}\label{dclem1}
Let $0\leq \lambda\leq 1,0<a\leq 3/2, V\subseteq 2^n$
and let $\Delta\lambda<1/6$.
Suppose $\wt(Q)\geq \lambda$ and
 $\wt(Q(V,a))\geq \lambda+\Delta\lambda$.
 Suppose there does not exists a desired tuple
 forcing $\mcal{R}_{\Psi,m^*}$ in a $\Pi_1^0$ way.
Then for every $\delta>0$,
there exists a finite set $V'\subseteq 2^{<\omega}$ with
$\m(V')<\frac{2}{3}\m(V)$
such that $\wt(Q(V',a+2\Delta\lambda))> \lambda+\Delta\lambda-\delta$.

\end{lemma}
\begin{proof}
By compactness, there exist
 $t,n\in\omega$ so that
for every $\rho\in Q(V,a)[t]\uhr_n$,
$\wt(\rho)>\lambda+\Delta\lambda-\delta$
and for every $\rho\in Q[t]\uhr_n$, $\wt(\rho)> \lambda-\delta$.
Fix a $\rho\in 2^n$ with $\lambda-\delta< \wt(\rho)\leq \lambda+\Delta\lambda-\delta$
(which means $[\rho]\cap Q(V,a)=\emptyset$).
We firstly prove that for every
$2$-partition $V_0,V_1$ of $[V]^\preceq\cap 2^{n'}$,
there exists an $i\in 2$ such that $\wt( [\rho]\cap Q(V_i,a+2\Delta\lambda))
\geq \lambda+\Delta\lambda$ (where $n'>n$ is sufficiently large).
Since $[\rho]\cap Q(V,a)=\emptyset$,
so $\wt([\rho]\cap Q(V,a))>3/2$.
By item (3) of Lemma \ref{dclem0},
$$\wt([\rho]\cap Q(V_0\cup V_1,a))\geq 3/2
\geq \lambda+3\Delta\lambda.$$
By Lemma \ref{dclem0} item (2),
we have that for some $i\in 2$,
$\wt([\rho]\cap Q(V_i,a+2\Delta\lambda))\geq \lambda+\Delta\lambda$.
Let $n'$ be sufficiently large
and for every $\rho\in 2^n$
with $\lambda-\delta<\wt(\rho)\leq \lambda+\Delta\lambda-\delta$,
let
\begin{align}\nonumber
\mcal{A}_\rho =
\{
V'\subseteq 2^{n'}:
\wt([\rho]\cap Q(V',a+2\Delta\lambda))\geq \lambda+\Delta\lambda
\}
\end{align}
We've shown that $\mcal{A}_\rho$ is $2$-large
in $2^{n'}\cap [V]^\preceq$.
By Lemma \ref{dclem0} item (1), $\mcal{A}_\rho$
is upward closed.
By Corollary \ref{dclemcomb2}, there exists
a $V'\subseteq 2^{n'}$ with
$\m(V')<\frac{2}{3}\m(V)$
such that $V'\in \bigcap\limits_{\rho\in 2^n,
\lambda-\delta<\wt(\rho)\leq \lambda+\Delta\lambda-\delta}\mcal{A}_\rho$.
Meanwhile, for every $\rho$ with $\wt(\rho)\leq \lambda-\delta$,
if $X'\succ \rho$, then
there exists no $X\in Q$ such that $X'\supseteq X$
since otherwise $\wt(X\uhr_n)\leq \lambda-\delta$,
a contradiction to our setting for $n$.
This means
$[\rho]\cap Q(V',a+2\Delta\lambda)=\emptyset$
if $\wt(\rho)\leq \lambda-\delta$.
Thus
$\wt([\rho]\cap Q(V',a+2\Delta\lambda))> \lambda+\Delta\lambda-\delta$
for all $\rho\in 2^n$, which means
$\wt(Q(V',a+2\Delta\lambda))> \lambda+\Delta\lambda-\delta$.

\end{proof}
Now fix
$0<\underline{a}<\overline{a}<3/2$,
a sufficiently small $1/6>\Delta\lambda>0$
and an $r\in\omega$
be such that:
\begin{align}
\left\{
\begin{aligned}\label{dceq0}
&\underline{a}+2(r+1)\Delta\lambda<\overline{a};
\\
&\frac{2}{\Delta\lambda}(\frac{2}{3})^r<\varepsilon.
\end{aligned}
\right.
\end{align}

\begin{lemma}\label{dclem3}
Let $0\leq \lambda\leq 1$.
Suppose $\wt(Q)\geq \lambda$
 and
 suppose there does not exist a desired tuple
 forcing $\mcal{R}_{\Psi,m^*}$ in a $\Pi_1^0$ way.
 Then there exists a $V'\subseteq 2^{n'}$ for some
 $n'$ with $\m(V')<(\frac{2}{3})^r$ such that
  $\wt(Q(V',\overline{a}))> \lambda+\Delta\lambda/2$.
\end{lemma}
\begin{proof}
Firstly, by Lemma \ref{dclem0} item (4),
$Q(2^n,\underline{a})\geq \lambda+\Delta\lambda$.
Now repeatedly apply Lemma \ref{dclem1} for $r+1$ times
(where $\delta$ is sufficiently small)
we have that there exists a $V'\subseteq 2^{n'}$
for some $n'$
with $\m(V')<(\frac{2}{3})^r$
such that $\wt(Q(V',\underline{a}+2(r+1)\Delta\lambda))\geq \lambda+\Delta\lambda-
(r+1)\delta$.
Therefore $\wt(Q(V',\overline{a}))
>\lambda+\Delta\lambda/2$.
\end{proof}

\begin{proof}[Proof of Lemma \ref{dclemmain}]
Now we define the following c.e. set
$V^*$:
\begin{definition}\label{dcdef1}
At step $0$,
search for a $V_0\subseteq 2^{n'}$
(for some $n'$) with  $\m(V_0)<(\frac{2}{3})^r$
such that $\wt(Q(V_0,\overline{a}))> \Delta\lambda/2$
and enumerate $V_0$ into $V^*$
(if such $V_0$ is not found it never enters the next step).

At step $s\geq 1$, wait for a time $t$ so that
$\wt(Q[t])> s\Delta\lambda/2$
(if such time $t$ does not appear, do nothing;
and we say $V^* $ enters phase II at step $s$ if such $t$ exists);
then search for a $V_s\subseteq 2^{n'}$
(for some $n'$) with  $\m(V_s)<(\frac{2}{3})^r$
such that $\wt(Q(V_s,\overline{a}))> (s+1)\Delta\lambda/2$
and enumerate $V_s$ into $V^*$
(if such $V_s$ is not found it never enters the next step).

\end{definition}
Note that $V^*$ is indeed c.e. since by compactness,
for every $\Pi^0_1$ class $\h{Q}$,
$\wt(\h{Q})> \h{\lambda}$ is
a $\Sigma^0_1$ relation uniformly
in the index of $\h{Q}$ and $\h{\lambda}$.
Suppose there does not exists a desired tuple
 forcing $\mcal{R}_\Psi$ in a $\Pi_1^0$ way.
Then by Lemma \ref{dclem3},
 at each step $s\geq 1$, if $V^*$ enters
 phase II,
the set $V_s$ exists.

Therefore, if there does not exists a desired tuple
 forcing $\mcal{R}_{\Psi,m^*}$ in a $\Pi_1^0$ way,
the program ends up by entering a step $s$
and never found a $t$ so that
$\wt(Q[t])> s\Delta\lambda/2$,
which means $\wt(Q)\leq s\Delta\lambda/2$.
But entering step $s$ implies $\wt(Q(V^*,\overline{a}))> s\Delta\lambda/2$
and $\wt(Q)>(s-1)\Delta\lambda/2$
(since $V^*$ enters Phase II at step $s-1$).
Since $\wt(Q)<1$,
we have $$s\Delta\lambda/2
= (s-1)\Delta\lambda/2+\Delta\lambda/2
<1+\Delta\lambda/2<3/2.$$
Therefore
there exists an $X\in Q\setminus Q(V^*,\overline{a})$
with $\wt(X)<3/2$.
By definition of $Q(V^*,\overline{a})$,
there exists a $\tau^*\supseteq X\uhr_{|\tau^*|}$ with
$\wt(\tau^*\setminus X)<\overline{a}$
such that $(\Psi^{\tau^*}\uhr_{n^*})\downarrow\in V^*$
for some $n^*$.

Meanwhile, since
 $\wt(Q)<1$,
 the algorithm of $V^*$ enters at most $\frac{2}{\Delta\lambda}$
 many steps
 and at each step,
  $\m(V^*)$ is increased by less
 than $(\frac{2}{3})^r$.
Therefore,
by (\ref{dceq0}),
 $$\m(V^*)<\frac{2}{\Delta\lambda}(\frac{2}{3})^r<\varepsilon.$$
Thus
 $(\tau^*,[X\uhr_{|\tau^*|}]\cap Q,\overline{a},V^*)$
is the desired tuple
since $$\wt([X\uhr_{|\tau^*|}]\cap Q)+\overline{a}\leq
\wt(X)+\overline{a}<
3/2+3/2=3.$$

\end{proof}

\section{The weakness of majorizing
bounded martingale}
\label{dcsec1}

Kastermans, Lempp, and Miller
showed that $D\in $High(CR,MLR)
if and only if there exists a $D$-computable
martingale $S$ such that $S$ succeed on
every non-1-random real
(see also \cite{bienvenu2012randomness}).
Therefore, in order to separate PA completeness
from High(CR,MLR), it suffices (and necessary)
to construct a martingale $S$ succeeding
on all non-1-random while does not compute
a PA degree.
A function $S$ is bounded if $range(S)$
is a bounded subset of $\mathbb{R}$.
In this section,
the domain of a martingale
(supermartingale resp) is $2^{<\omega}$
if not claimed otherwise.
A tree $T
\subseteq 2^{<\omega}$ is \emph{homogeneous}
if
there exists a sequence of
set $B_0,B_1\cdots\subseteq \{0,1\}$ such that
$\rho\in T$ if and only if $\rho(n)\in B_n$
for all $n<|\rho|$.

\begin{theorem}\label{dcth0}
Let
$\overline{\S}_0,\overline{\S}_1\cdots$ be
 a sequence of bounded martingales
(not necessarily c.e.) such that
$\lim\limits_{n\rightarrow\infty}\overline{S}_n(\bot)=0$;
let $T\subseteq 2^{<\omega}$ be a homogeneous
tree that does not admit a computable infinite path.
There exists a martingale
$S$ such that
$\S\geq \overline{S}_n$ for infinitely many $n$
and $\S$ does not compute
an infinite path through $T$.
\end{theorem}

Note that  every open set $U$ of $2^\omega$
give rise to a bounded martingale
$\S_U$ such that $\S_U(\rho) = \m(U\cap [\rho])/\m([\rho])$.
For every sequence of open set $U_n$ with
$\m(U_n)\leq 1/2^n$,
consider the induced bounded martingale
sequence $n\cdot S_{U_n}$.
Clearly $\lim\limits_{n\rightarrow\infty}
n\cdot S_{U_n}(\bot)=0$.
Thus by Theorem \ref{dcth0}
with $\overline{S}_n =
n\cdot S_{U_n}$
where $(U_n:n\in\omega)$
is the universal ML-test
and with $T$ be the homogenous tree
defining PA degree, we directly have:

\begin{corollary}\label{dccoro2}
There exists a $D\in$ High(CR,MLR)
such that $D$ is not of PA degree.

\end{corollary}
Another corollary is that 
$\leq_{\msf{CR}}$ does not imply 
$\leq_T$.

\begin{corollary}\label{dccoro1}
For every incomputable set $C$, there exists
a $D\ngeq_T C$ such that $\msf{CR}^D\subseteq \msf{MLR}^C$.
\end{corollary}
\begin{proof}
Let $(U_n^C:n\in\omega)$ be the universal ML-test relative to
$C$ and let $\overline{S}_n = n\cdot S_{U^C_n}$;
let $T = \{C\uhr_n\}_{n\in\omega}$.
Thus the conclusion follows directly from Theorem \ref{dcth0}.

\end{proof}

The bounddeness restriction in Theorem \ref{dcth0}
might not be necessary but is for the sake of convenience
(see the end of section \ref{dcsecprob}).
Now we prove Theorem \ref{dcth0}.
Let's first recall some probability notions.
The main object of section \ref{dcsecprob}
is Lemma \ref{dclem8}.
\subsection{Majorize
multiple martingales with small variance}
\label{dcsecprob}
We address the question that
given $k$
martingales $S_0,\cdots,S_{k-1}$, a
$c>\max\limits_{j<k}\{S_j(\bot)\}$,
under what condition does there guarantee a
martingale $S^*$ majorizing  $\max\limits_{j<k}\{S_j(\rho)\}$
while $S^*(\bot)<c$.
The condition $c>\sum\limits_{j<k}S_j(\bot)$
is too strong for our application.
For a martingale
$\S$, a $\rho\in 2^{<\omega}$,
let
$$
Var(S|\rho) =
\lim\limits_{t\rightarrow\infty}
\frac{1}{2^{t-|\rho|}}
\sum\limits_{\sigma\in [\rho]^\preceq\cap 2^t}
(S(\sigma)-S(\rho))^2.
$$
We show that a sufficient condition is that
$c>\max\limits_{j<k}\{S_j(\bot)\}+
\frac{1}{2}\sum\limits_{j<k}\sqrt{Var(S_j|\bot)}$.

If $S$ is martingale,
by variance decomposition we have:
$$
Var(S|\rho)
= \frac{1}{2}(Var(S|\rho0)
 +Var(S|\rho1))
 +\frac{1}{2}
 ((S(\rho0)-S(\rho))^2+
 (S(\rho1)-S(\rho))^2).
$$
i.e., $Var(S|\cdot)$
(as a function $\rho\mapsto\mathbb{R}^{\geq 0}$)
is a supermartingale.
Also note that if $Var(S|\bot)<\infty$, then
\begin{align}\label{dceq2}
&\lim\limits_{t\rightarrow\infty}
\frac{1}{2^{t-|\rho'|}}
\sum\limits_{\rho\in 2^t\cap [\rho']^\preceq}Var(S|\rho)
=0 \text{ for all $\rho'$}.
\end{align}

Since $\sqrt{x}$
is concave, for every non negative random variable
$x$, $\sqrt{\mathbb{E}[x]}
\geq \mathbb{E}[\sqrt{x}]$.
Combine with
 (\ref{dceq2}),
we have that:
\begin{align}\label{dceq3}
&%\sqrt{Var(S|\cdot)}\text{ is a supermartingale.}\\ \nonumber
\lim\limits_{t\rightarrow\infty}
\frac{1}{2^{t-|\rho'|}}
\sum\limits_{\rho\in 2^t\cap [\rho']^\preceq}\sqrt{Var(S|\rho)}
=0
\text{ for all $\rho'$}.
\end{align}

\begin{lemma}\label{dclem4}
Given
a martingale $S'$ with $Var(S'|\bot)<\infty$
and a $c> S'(\bot)$.
For every $\h{c}$, there exists a martingale
$S$ such that:
\begin{itemize}
\item $S(\bot)< c$
and $S(\rho)>
S'(\rho)$
for all $\rho\in 2^{<\omega}$;
\item For some $t\in\omega$,
$S(\rho)>
S'(\rho)
+\h{c}\cdot\sqrt{Var(S'|\rho)}$
for all $\rho\in 2^{t}$.

\end{itemize}
\end{lemma}

\begin{proof}
It's trivial to see that
a martingale $S$ satisfying the first  item
exists.
Suppose
$S(\bot)< c-\delta$ for some $\delta>0$.
By (\ref{dceq3}),
there exists a $t$ such that
$$\frac{\h{c}}{2^t}\sum\limits_{\rho\in 2^t}
\sqrt{Var(S_j|\rho)}<\delta.$$
Let $\h{S}$ be the unique martingale such
that for every $\rho\in 2^t$
and $\sigma\in [\rho]^\preceq$,
$\h{S}(\sigma)= \h{S}(\rho) =
\h{c}\cdot\sqrt{Var(S_j|\rho)}$.
Clearly $\h{S}(\bot)<\delta$.
Thus $\h{S}+S$
satisfy all items.

\end{proof}

\begin{lemma}\label{dclem9}
For any random vector $\v{x}$
(not necessarily with mutually independent component),
\begin{align}\label{dceq5}
\mathbb{E}[\max\limits_{j<k}\{\v{x}(j)\}]-\max\limits_{j<k}\{\mathbb{E}[\v{x}(j)]\}\leq
\frac{1}{2}\sum\limits_{j<k}\sqrt{Var(\v{x}(j))}.
\end{align}
\end{lemma}
\begin{proof}
Without loss of generality,
we may assume that $\mathbb{E}[\v{x}(j)] = \mu$
for all $j<k$.
Let $A_j$ denote the event
$\{\v{x}(j) = \max\limits_{j'<k}\{\v{x}(j')\}\wedge
\v{x}(j)>\v{x}(j')\text{ for all }j'<j\}$.
Note that $\cup_j A_j$ is the whole space.
Suppose $\mathbb{E}[\v{x}(j)|A_j] = \mu+a_j$
and $\mathbb{E}[\v{x}(j)|\overline{A}_j] = \mu-b_j$.
Clearly $\mathbb{P}(A_j) = \frac{b_j}{a_j+b_j}$.
By decomposition of variance:
\begin{align}\nonumber
Var(\v{x}(j)) &= \mathbb{P}(A_j)
[Var(\v{x}(j)|A_j)+ a_j^2]+
\mathbb{P}(\overline{A}_j)
[Var(\v{x}(j)|\overline{A}_j)+ b_j^2]\\ \nonumber
&\geq \frac{b_j}{a_j+b_j}a_j^2+
\frac{a_j}{a_j+b_j}b_j^2\\ \nonumber
&=a_jb_j
\end{align}

Meanwhile, since $\frac{\sqrt{a_jb_j}}{a_j+b_j}\leq 1/2$,
\begin{align}\nonumber
\mathbb{E}[\max\limits_{j<k}\{\v{x}(j)\}]
-\max\limits_{j<k}\{\mathbb{E}[\v{x}(j)]\}&=
\sum\limits_{j<k}\mathbb{P}(A_j)a_j\\ \nonumber
&=\sum\limits_{j<k} \frac{\sqrt{a_jb_j}}{a_j+b_j}\sqrt{a_jb_j}
\\ \nonumber
&\leq \frac{1}{2}\sum\limits_{j<k}\sqrt{Var(\v{x}(j))}.
\end{align}

\end{proof}
The following Lemma \ref{dclem7}
is key to our proof.
\begin{definition}
For vectors $\v{\mu}\in\mathbb{R}^k,
\v{v}\in (\mathbb{R}^{\geq 0})^k$,
define 
$
f(\v{\mu},\v{v})$ to be the supreme of
$
\mathbb{E}[\max\limits_{j<k}\{\v{x}(j)\}]$
where the supreme is taken over all random vector
$\v{x}$ such that
$\mathbb{E}[\v{x}] = \v{\mu}$ and
$Var(\v{x}(j)) = \v{v}(j)^2$ for all $j<k$.

\end{definition}

\begin{lemma}\label{dclem7}
For every
$\v{\mu}\in\mathbb{R}^k,
\v{v}\in (\mathbb{R}^{\geq 0})^k$ we have:
\begin{enumerate}

\item $\max\limits_{j<k}\{\v{\mu}(j)\}\leq f(\v{\mu},\v{v})\leq
\max\limits_{j<k}\{\v{\mu}(j)\}+\frac{1}{2}
\sum\limits_{j<k}\v{v}(j)$;
%where $c(k)$ is a constant depending on $k$;
\item Suppose $\v{\mu}_i,\v{v}_i,i\in 2$
satisfy $\v{\mu} = \frac{1}{2}
\sum\limits_{i\in 2}\v{\mu}_i$,
$\v{v}(j)^2= \frac{1}{2}
\sum\limits_{i\in 2}
\big[\v{v}_i(j)^2+
(\v{\mu}_i(j)-\v{\mu}(j))^2\big]$
for all $j<k$.
Then we have:
$\frac{1}{2}(f(\v{\mu}_0,\v{v}_0)+f(\v{\mu}_1,\v{v}_1))
\leq f(\v{\mu},\v{v})$.

%\item For every $\mu'\leq \max\limits_{j<k}\{\v{\mu}(j)\}$,
%$\lim\limits_{\delta\rightarrow 0}
%f((\v{\mu}(0),\cdots,\v{\mu}(k-1),\mu'),(\v{v}(0),\cdots,\v{v}(k-1),\delta))=
%f(\v{\mu},\v{v})$.
\end{enumerate}
\end{lemma}
\begin{proof}
For item (1): Note that for every random vector $\v{x}$,
it holds trivially that $\mathbb{E}[\max\limits_{j}\{\v{x}(j)\}]
\geq \max\limits_{j}\{\mathbb{E}(\v{x}(j))\}$,
which gives the first inequality.
The second inequality follows by (\ref{dceq5}).
%$$\mathbb{E}[\max\limits_{j<k}\{x_j\}]-\max\limits_{j<k}\{\mathbb{E}[x_j]\}\leq \frac{1}{2}\sum\limits_{j<k}\sqrt{Var(x_j)}.$$

For item (2):
Suppose otherwise with $\v{\mu}_i,\v{v}_i,\v{\mu},\v{v},i\in 2$
being a witness.
Let $\v{x}_i,i\in 2$ be random vectors such that
$\mathbb{E}(\v{x}_i) = \v{\mu}_i,
Var(\v{x}_i(j)) = \v{v}_i(j)$ and
$\mathbb{E}[\max\limits_{j}\{\v{x}_i(j)\}]
>f(\v{\mu}_i,\v{v}_i)-\delta$ for all $i\in 2$.
Where $\delta>0$ satisfy
$\frac{1}{2}(f(\v{\mu}_0,\v{v}_0)+f(\v{\mu}_1,\v{v}_1))
> f(\v{\mu},\v{v})+\delta$.
Let $z\sim Bin(1,1/2)$ where $Bin(n,p)$ is the Bernoulli
distribution
and let $\v{x}$ be such a random vector that
conditional on $z=i$, $\v{x}\sim Law(\v{x}_i)$ for all $i\in 2$
(where $Law(\v{x}_i)$ denote the distribution of $\v{x}_i$).
Clearly
\begin{align}\nonumber
\mathbb{E}[\max\limits_{j}\{\v{x}(j)\}]
 &= \frac{1}{2}
 \sum\limits_{i\in 2}\mathbb{E}[\max\limits_{j}\{\v{x}_i(j)\}]
 \\ \nonumber
 &>\frac{1}{2}(f(\v{\mu}_0,\v{v}_0)+f(\v{\mu}_1,\v{v}_1))-\delta
 \\ \nonumber
 &> f(\v{\mu},\v{v})+\delta-\delta\\ \nonumber
 &=f(\v{\mu},\v{v}).
 \end{align}
 Meanwhile,
 $\mathbb{E}[\v{x}] = \frac{1}{2}
 \sum\limits_{i\in 2}\mathbb{E}[\v{x}_i]
  = \frac{1}{2}
 \sum\limits_{i\in 2}\v{\mu}_i
 =\v{\mu}$,
 and
 \begin{align}\nonumber
 Var(\v{x}(j)) &=
 \frac{1}{2}\sum\limits_{i\in 2}Var(\v{x}_i(j))
 +\frac{1}{2}\sum\limits_{i\in 2}
 (\v{\mu}_i(j)-\v{\mu}(j))^2
  \\ \nonumber
  &= \frac{1}{2}
\sum\limits_{i\in 2}
\big[\v{v}_i(j)^2+
(\mathbb{E}[\v{x}_i(j)]-\mathbb{E}[\v{x}(j)])^2\big]
= \v{v}(j)^2
\end{align}

 A contradiction to the definition of
 $f(\v{\mu},\v{v})$.

\end{proof}

\begin{lemma}\label{dclem8}
Given $k$  martingales $S_0,\cdots,S_{k-1}$,
a $c>\max\limits_{j<k}\{S_j(\bot)\}+\frac{1}{2}
\sum\limits_{j<k}\sqrt{Var(S_j|\bot)}$,
there exists a  martingale
$S^*$ with $S^*(\bot)< c$ such that
$S^*(\rho)>\max\limits_{j<k}\{S_j(\rho)\}$
for all $\rho\in 2^{<\omega}$.
Moreover, $S^*$ could be bounded if each $S_j$ is bounded.
\end{lemma}
\begin{proof}
Let $\v{\S}=(\S_0,\cdots,\S_{k-1})$,
 we write $\v{S}(\rho)$
for $(\S_0(\rho),\cdots,\S_{k-1}(\rho))$;
we write $\sqrt{Var}(\v{S}|\rho)$
for
$(\sqrt{Var}(S_0|\rho),\cdots,
\sqrt{Var}(S_{k-1}|\rho))$.
Consider the function
$\rho\mapsto f(\v{S}(\rho),\sqrt{Var}(\v{S}|\rho))$.
By Lemma \ref{dclem7} item (2),
$f(\v{S}(\rho),\sqrt{Var}(\v{S}|\rho))$
is a supermartingale.
Moreover, by Lemma \ref{dclem7} item (1),
$f(\v{S}(\rho),\sqrt{Var}(\v{S}|\rho))$
is bounded if $S_j$ is bounded for each $j<k$
(note that $S$ is bounded implies
$Var(S|\rho)$ is bounded).
By Lemma \ref{dclem7} item (1),
$f(\v{S}(\bot),\sqrt{Var}(\v{S}|\bot))
<c$. Thus there exists
a martingale $S^*$ such that
$S^*(\bot)<c$ and
$S^*(\rho)>f(\v{S}(\rho),\sqrt{Var}(\v{S}|\rho))$
for all $\rho\in 2^{<\omega}$.
Moreover, $S^*$ could be bounded if
$S_j$ is bounded for each $j<k$.
Meanwhile, by Lemma \ref{dclem7} item (1),
$S^*(\rho)>f(\v{S}(\rho),\sqrt{Var}(\v{S}|\rho))
\geq \max\limits_{j<k}
\{S_j(\rho)\}$ for all $\rho\in 2^{<\omega}$.
Thus we are done.

\end{proof}

To remove the boundedness restriction in Theorem \ref{dcth0},
we need to prove that $S^*$ could have finite variance
without the boundedness of $S_0,\cdots,S_{j-1}$.
This is entirely possible.
\subsection{Proof of Theorem \ref{dcth0}}
\label{dcsecproofth0}
Again, the frame work of the proof follows the usual
line in computability theory. We build a sequence of
conditions each forces some requirement and
thus the martingale $\h{S}$
we construct will satisfy all needed requirements.
We firstly define the condition we use then
we focus on how to force a given requirement
(Lemma \ref{dclemmain}).
Readers who  are familiar with representation
of reals and martingales can skip the following
introduction on this representation.

For every $x\in\mathbb{R}^{\geq 0}$,
we represent $x$
as an $X=(b,a_0a_1\cdots)
\in \omega\times 2^{\omega}\cup
\omega\times 2^{<\omega}$
so that: 1)
 $x=\sum\limits_{|X|> n\geq 1}X(n)\cdot 2^{X(0)-n}$
 where $X(0) = b, X(n) = a_{n-1}$;
 2) $b>0\rightarrow a_0=1$.
 The second item means that $1/2$ cannot be represented
 as $(2,00100\cdots)$ or $(1,0100\cdots)$.
 Conversely every element of $\omega\times 2^{\omega}\cup
\omega\times 2^{<\omega}$ is seen as a real.
Meanwhile, given a real, its representation is not unique.
For example, $1/2$ can be represented as
$(0,100\cdots)$ and $(0,011\cdots)$.
An initial segment of
a $X\in \omega\times 2^\omega$
with length $n$
is a $\rho\in \omega\times 2^n$
such that $\rho =
X\uhr_n$.
An initial segment of
a function $\S:2^{<\omega}\rightarrow \omega\times2^\omega$
 is a function $s$ with $dom(s)=2^{\leq n} $ for some
 $n$
 such that for every $\rho\in 2^{\leq n}$, $s(\rho)$
 is an initial segment of $\S(\rho)$.
 We still use $\succeq$
  to denote extension relation in these
   spaces.
Clearly
\begin{itemize}
\item $\{Y\in \omega\times 2^\omega:
X_1\geq Y\geq X_0\}$
is a $\Pi_1^{0,X_0\oplus X_1}$ class
uniformly in $X_0\oplus X_1$;

\item
$\{\h{S}:
\h{S}\text{ is a martingale and }
\h{\S}\leq Y\wedge Var(\h{S}|\bot)\leq X\}$
is a $\Pi_1^{0,X\oplus Y}$ class uniformly in $X\oplus Y$.

\item
$\{\h{S}:
\h{S}\text{ is a martingale and }
\h{\S}\succ s\wedge
\h{S}\geq S\}$
is a $\Pi_1^{0,\S}$ class
uniformly in $s, \S$
where $S$
is a function $2^{<\omega}
\rightarrow \omega\times 2^\omega$,
$s$ is an initial segment of some function.

\item If $X$ is the representation of an irrational
and $Y\geq X$,
then for every $\sigma\prec Y$, there exits
a $\tau\succeq\sigma$ such that
$\tau>X$.

\item For every $s$ that is an initial segment of
a martingale with $dom(s) = 2^{\leq n}$, every martingale $S$
with $S(\rho)<s(\rho)$ for all $\rho\in 2^n$,
there exists a martingale $\h{\S}\succ s$ such that
$\h{\S}(\rho)>\S(\rho)$ for all $\rho \in 2^{\geq n}$.

\item For every $s\prec \S$, there exists a $\delta>0$
such that for every function $\h{\S}$, if
$\S(\rho)+\delta\geq \h{\S}(\rho)\geq \S(\rho)$
for all $\rho\in dom(s)$, then $\h{\S}\succ s$.

\item Suppose $S$ is a martingale with $range(S)\cap \mathbb{Q}=\emptyset$,
and suppose $s$ is an initial segment of some martingale
with $s(\rho)>S(\rho)$ for all $\rho\in 2^n$
where $dom(s) = 2^{\leq n}$; suppose
$\h{\S}\succ s$ is a martingale such that
$\h{\S}(\rho)\geq \S(\rho)$ for all $\rho\in 2{\geq n}$.
Then for every $s\preceq s'\prec \h{\S}$,
there exists a an initial segment $\h{s}\succeq s'$
 of some martingale (not necessarily an initial segment of
 $\h{\S}$) such that $\h{s}(\rho)>S(\rho)$ for all
 $\rho\in 2^{\geq n}\cap dom(\h{s})$.
\end{itemize}

A \emph{condition}
in this section is a pair
$(s,\S)$
such that
\begin{itemize}
\item $s$ is an initial segment of
some martingale with $dom(s) = 2^{\leq n}$.
\item $\S$
is a
 bounded martingale
%with $S\leq c$ for some $c$
with $range(\S)
\cap\mathbb{Q}=\emptyset$.

\item For every $\rho \in 2^n$,
$s(\rho)>S(\rho).$
\end{itemize}
We emphasize that there is no complexity requirement
on $\S$.
Again, a condition is regarded as the following set of
martingales:
\begin{align}\nonumber
\{\h{S}\succ s:
\h{\S}(\rho)\geq \S(\rho)
\text{ for all }\rho\in 2^{\geq n}
\}.
\end{align}
This is a set of candidates of
the martingale we construct
and we use $(s,S)$
to denote this set.
A condition
$d'=(s',S')$
\emph{extends}
$d=(s,\S)$
(written as $d'\leq d$)
if
$
d'\subseteq d
$.
Suppose $B_0,B_1,\cdots\subseteq \{0,1\}$ witness
$T$ being homogenous.
The requirements are:
\begin{align}\nonumber
&\mcal{R}_\Psi:
\text{For some }n, \Psi^{\h{S}}(n)\downarrow\notin B_n
\text{ or }
\Psi^{\h{\S}}\text{ is not total}. \\ \nonumber
&\mcal{R}'_n:
\text{There exists }n'>n
\text{ such that }\h{\S}\geq \overline{S}_{n'}.
\end{align}
A condition $d$
\emph{forces} $\mcal{R}_\Psi$
if every $\h{\S}\in d$ satisfy $\mcal{R}_\Psi$.
By our hypothesis on $(\overline{S}_0,\overline{S}_1,\cdots)$,
for every $n$,
every condition $d=(s,S)$,
$d$ admit an extension forcing
$\mcal{R}'_{n}$
(simply consider the condition
$(s,\S+\overline{S}_{n'})$) where
$\overline{S}_{n'}(\bot)$ is small enough.
Therefore, it remains to prove the following:
\begin{lemma}\label{dclem10}
Every condition admit an extension
forcing a given requirement $\mcal{R}_\Psi$.
\end{lemma}

Using Lemma \ref{dclem10}, we can build
a sequence of  conditions $d_0\geq d_1\geq \cdots$
where $d_t=(s_t,S_t)$ such that every
requirement is forced by one of these conditions.
By compactness, $G=\cup_t s_t$ is
a martingale contained in all conditions and
 thus satisfy all requirements.

\begin{proof}

Fix a condition $(s,S)$.
The basic idea of this proof follows
a proof separating $\msf{WKL}$
from $\msf{RT}_2^2$
 \cite{Liu2012RT22},
 illustrated as following.
If for some $m^*$, $B_{m^*}=\{i^*\}$ and no condition 
can force $\neg(\Psi^{\h{S}}(m^*)\downarrow= 1-i^*)$, then in particular, 
$(s,S)$ can not force this. We  show that in this case
we can extend $s$ to $\h{s}$
(which does not violate the condition $(s,S)$) so that $\Psi^{\h{s}}(m^*)\downarrow
=1-i^*$.
If  for each $i\in \{0,1\}$,
there is a condition $(s,S_i)$ forcing $
\neg(\Psi^{\h{S}}(m^*)\downarrow= i)$, then
we consider the three condition $(s,S),(s,S_0), (s,S_1)$
and show that there exists a condition $(s,S^*)$
extending all of them
(simply let $S^*$ be a martingale majorizing $S,S_0,S_1$ which exists 
by Lemma \ref{dclem8}). 
Note that for every $\h{S}\in (s,S^*)$, $\Psi^{\h{S}}(m^*)\uparrow$
since otherwise, either $\Psi^{\h{S}}(m^*)\downarrow=0$
or $\Psi^{\h{S}}(m^*)\downarrow=1$.

Suppose
$dom(s) = 2^{\leq n}$.
By Lemma \ref{dclem4},
we  may assume that, for some $0<\delta\in\mathbb{Q}$,
for
every $\rho\in 2^n$,
\begin{align}\label{dceq4}
s(\rho)>
\S(\rho)+3/2\cdot
v(\rho)+\delta.
\end{align}
where $\mathbb{Q}\ni v(\rho)>\sqrt{Var(\S|\rho)}$
for all $\rho\in 2^n$.
For every $m\in\omega$,
every $i\in 2$,
consider the following  class
$Q_{m,i}$ such that
$\S'\in Q_{m,i}$ if and only if:
\begin{enumerate}
\item $\S'$ is a martingale;
\item $\S'(\rho)\leq s(\rho)-3/2\cdot v(\rho)-\delta$
and $\sqrt{Var(\S'|\rho)}\leq
v(\rho)$ for all $\rho\in 2^n$;
 \item
 For  every $\h{S}\in (s,\S')$,
$\Psi^{\h{\S}}(m)\downarrow\rightarrow
\Psi^{\h{\S}}(m)\ne i$.
\end{enumerate}
Note that $Q_{m,i}$
is a $\Pi_1^0$ class uniformly in
$m,i$.
Consider the following partial computable
function:
$$g(m) = \text{
the first $i\in 2$
which is found that $Q_{m,i}[t]=\emptyset$
for some
}t.$$
Since $T$ does not admit a computable infinite path,
one of the following two cases occurs.

\ \\

\textbf{Case 1}:
There exists an $m^*\in\omega$, $g(m^*)\notin B_{m^*}$.

This means
$Q_{m^*,i^*}=\emptyset$
where $g(m^*) = i^*\notin B_{m^*}$.
In this case,
in particular, $
\S\notin Q_{m^*,i^*}$
while
$\S$
satisfy the first two items of $Q_{m^*,i^*}$.
By definition of $Q_{m,i}$,
there exists a martingale $\h{S}\succ s$
such that
\begin{itemize}
\item $\h{\S}(\rho)\geq
\S(\rho)$ for all $\rho\in 2^{\geq n}$;
\item $\Psi^{\h{\S}}(m^*)\downarrow=i^*$.
\end{itemize}
Since $range(\S)\cap \mathbb{Q}=\emptyset$,
there exists an initial segment $\h{s}\succeq s$
of some martingale such that
$\h{s}(\rho)>\S(\rho)$ for all $\rho\in 2^{\geq n}\cap dom(\h{s})$
(which means $(\h{s},\S)$
is a condition extending $(s,\S)$);
  and $\Psi^{\h{s}}(m^*)\downarrow=i^*\notin B_{m^*}$.
Thus we are done in this case.

\ \\

\textbf{Case 2}:
There exists an $m^*$ such that
$g(m^*)\uparrow$.

This means
$Q_{m^*,0},Q_{m^*,1}\ne\emptyset$.
Fix
$S_i\in Q_{m^*,i}$
for $i\in 2$
(not necessarily low).
Consider the sequence
$\S,S_0,S_1$.
By (\ref{dceq4}) and the definition of $Q_{m,i}$,
for every $\rho\in 2^n$,
$$s(\rho)>
\max\{\S(\rho),S_0(\rho),S_1(\rho)\}
+\frac{1}{2}
(\sqrt{Var(\S|\rho)}+
\sqrt{Var(S_0|\rho)}+\sqrt{Var(S_1|\rho)}).$$
By Lemma \ref{dclem8},
there exists
a martingale $S^*$ such that
\begin{itemize}
\item $S^*$ is a boundeded martingale
with $range(S^*)\cap\mathbb{Q}=\emptyset$;
\item $S^*(\rho)<s(\rho)$ for all $\rho\in 2^n$;
\item $S^*(\rho)> \max\{\S(\rho),S_0(\rho),S_1(\rho)\}$
for all $\rho\in 2^{\geq n}$.
\end{itemize}
Thus the pair
$(s,S^*)$ is a condition.
Clearly
$(s,S^*)$ extends $(s,\S)$
since $S^*(\rho)
\geq \max\{\S(\rho),\S_0(\rho),\S_1(\rho)\}
\geq \S(\rho)$
for all $\rho\in 2^{\geq n}$.
It remains to show that $(s,S^*)$
forces $\mcal{R}_\Psi$.
For every $\h{S}\in (s,S^*)$,
we have that $\h{S}\succ s$ and
$\h{\S}(\rho)\geq \S^*(\rho)\geq S_i(\rho)$ for all $\rho\in 2^{\geq n}$
(which means $\h{S}\in (s,\S_i)$).
By definition of $Q_{m,i}$,
we have that $\Psi^{\h{\S}}(m^*)\downarrow\rightarrow
\Psi^{\h{\S}}(m^*)\notin \{0,1\}$.
Which means $\Psi^{\h{\S}}(m^*)\uparrow$.
Thus we are done.

\end{proof}

\bibliographystyle{amsplain}
\bibliography{bibliographylogic}

\end{document}